\documentclass[11pt]{amsart}

\usepackage{amsmath,amsthm}
\usepackage{amsfonts}
\usepackage{verbatim}
\usepackage{amssymb}
\usepackage{amscd}
\usepackage{bbm}

\usepackage{mathrsfs}
\usepackage[nohug]{diagrams}

\usepackage[nobysame,non-compressed-cites]{amsrefs}
\usepackage[citecolor=red,colorlinks=true]{hyperref}%citecolor=cyan,linkcolor=cyan,backref=page
\usepackage{fouriernc}

\newcommand{\set}[1]{\,\left\{#1\right\}}%set
\newcommand{\setd}[2]{\,\left\{#1\ \colon\ #2\right\}}%set with a decription usage\set{elements}

\newtheorem{theorem}{Theorem}
\newtheorem{corollary}[theorem]{Corollary}
\newtheorem{lemma}[theorem]{Lemma}
\newtheorem{proposition}[theorem]{Proposition}
\newtheorem{definition}[theorem]{Definition}

\newtheorem{remark}[theorem]{Remark}
\newtheorem{example}[theorem]{Example}

\newcommand{\ZZ}{\mathbb{Z}}

\newcommand{\NN}{\mathbb{N}}

\newcommand{\F}{\operatorname{F}}
\newcommand{\Hred}{\overline{H}}
\newcommand{\Hom}{\operatorname{Hom}}
\newcommand{\im}{\operatorname{im}}
\newcommand{\Ann}{\operatorname{Ann}}

\title{Cohomology of deformations}

\thanks{The first author was partially supported by the European Research Council. 
The second author was partially supported by the Foundation for Polish Science}

\author{Uri Bader}
\address{Technion, Haifa, Israel}
\email{uri.bader@gmail.com}

\author{Piotr W. Nowak}
\address{Instytut Matematyczny Polskiej Akademii Nauk, Warsaw, Poland}
\address{Instytut Matematyki, Uniwersytet Warszawski, Warsaw, Poland}
\email{pnowak@mimuw.edu.pl}
\begin{document}

\maketitle

\begin{abstract}
In this article we study  cohomology of a group with coefficients in representations on Banach spaces 
and its stability under deformations.
We show that small, metric deformations of the representation preserve vanishing of cohomology.
As applications we obtain deformation theorems for fixed point properties on Banach spaces. In particular, our results
yield fixed point theorems for affine actions in which the linear part is not uniformly bounded.
Our proofs are effective and allow  for quantitative estimates.
\end{abstract}

Cohomology with coefficients in  representations on Banach spaces is a broad topic, encompassing important notions like fixed point properties and Kazhdan's property (T),
amenability, a-T-menability or the Haagerup property, and their generalizations, as well as $\ell_p$-cohomology. 
The main theme in this paper is how does cohomology of a group with coefficients in a representation on a Banach space
change under small, metric perturbations of that representation.
Let $\Gamma$ be a group generated by a finite set $S$. Given a representation of $\Gamma$ on a Banach  space $E$ and $\varepsilon>0$, 
a representation $\rho$ of $\Gamma$ on the same Banach space 
$E$  is said to be an \emph{$\varepsilon$-deformation} of $\pi$ if 
$$\sup_{s\in S}\Vert \pi_s-\rho_s\Vert_{B(E)}\le \varepsilon.$$
We are interested in the stability of the behavior of cohomology groups under such deformations of representations.
In degree 1, $\pi$-cocycles correspond to affine actions with linear part $\pi$ and in this case the question we are interested in is whether
for an $\varepsilon$-deformation $\rho$  of $\pi$ as above, properties of the affine $\pi$-actions influence the properties of affine $\rho$-actions.
The main result we prove is a deformation principle for vanishing of cohomology with coefficients in Banach modules.

\begin{theorem}\label{theorem: main theorem}
Let $\Gamma$ be a group of type $\F_{n+1}$, $n\ge 1$, with a corresponding Eilenberg-MacLane space $X$
and let $\pi$ be a representation of $\Gamma$ on a Banach space $E$.
Assume that
\begin{enumerate}
\item $H^n(\Gamma,\pi)=0$,
\item $H^{n+1}(\Gamma,\pi)$ is reduced.
\end{enumerate}
Then there exists a constant $\varepsilon=\varepsilon(\Gamma,X,\pi,n)>0$, such that 
for every $\varepsilon$-deformation $\rho$  of $\pi$ we have 
$$H^n(\Gamma,\rho)=0.$$
Additionally, the $n$-th Kazhdan constant can be estimated explicitly. 
\end{theorem}

Theorem \ref{theorem: main theorem} will follow from a more general version in which we consider 
the $L_p$-cohomology with coefficients twisted by $\pi$ for $1\le p < \infty$.

\begin{theorem}[Theorem 1$^\prime$]
Let $\Gamma$ be a group acting on a uniformly locally finite simplicial complex $X$, $\pi$ be a representation of $\Gamma$ on a Banach space $E$ and let  $1\le p<\infty$.
Assume that
\begin{enumerate}
\item $L_pH^n(X,\Gamma,\pi)=0$,
\item $L_pH^{n+1}(X,\Gamma,\pi)$ is reduced.
\end{enumerate}
Then there exists a constant $\varepsilon=\varepsilon(\Gamma,X,\pi,n,p)>0$, such that 
for every $\varepsilon$-deformation $\rho$  of $\pi$ we have 
$$L_pH^n(X,\Gamma,\rho)=0.$$
\end{theorem}

Note that we do not assume any uniform boundedness of either representation. However, even if $\pi$ is an isometric representation, 
an $\varepsilon$-deformation $\rho$ does not have to be uniformly bounded - all that matters is the distance between the generators. 
Natural examples of such transformations can be obtained by twisting a given representation
by an appropriately chosen cocycle. We discuss such examples in Section \ref{section : deformations}.

In \ref{subsection:H^n+1 must be reduced} we also discuss an example showing that the second assumption that the cohomology group 
$H^{n+1}(\Gamma,\pi)$ is reduced cannot be dropped.
As a by-product of our methods we also obtain a deformation principle for cocycles: if $\rho$ is an $\varepsilon$-deformation of $\pi$ for
a sufficiently small $\varepsilon>0$, then the cocycles for $\rho$ lie close, in an appropriate sense, to cocycles for $\pi$, provided
that $H^{n+1}(\Gamma,\pi)$ (or $L_pH^{n+1}(\Gamma,\pi)$) is reduced. See 
Theorem \ref{theorem: close representations have close cocycles} in the text for the precise formulation.

There are many instances in which the condition that $H^{n+1}(\Gamma,\pi)$ is reduced is satisfied. One example is when the representation $\pi$ is 
finite-dimensional, then automatically, $H^{n+1}(\Gamma,\pi)$ is reduced. This is clear in the case when $\Gamma$ is of type $F_{n+1}$, see also \cite{austin} for the general case.
Another natural case is when $H^{n+1}$ vanishes. Examples of vanishing theorems for higher-dimensional cohomology with coefficients in unitary representations
can be found in \cite{borel-wallach,dymara-januszkiewicz} and include ``vanishing up to the rank'' phenomena and vanishing for automorphism groups of thick buildings.
Interestingly, the class of groups for which the assumptions of Theorem \ref{theorem: main theorem}
are \emph{always} satisfied is related to higher-dimensional analogs of Kazhdan's property (T). We discuss this phenomenon in Section \ref{subsection:higher T}.

The higher Kazhdan constants, mentioned in the above theorems, are defined in section \ref{section: higher Kazhdan constants}. They 
are higher dimensional analogs of the usual Kazhdan constant for 
groups with property (T). We 
define the $n$-th Kazhdan constant in the setting of cohomology, under the assumption that the cohomology in degree $n+1$ is reduced. 

The main application and the original motivation for Theorem \ref{theorem: main theorem} are fixed point properties for actions of discrete groups on 
Banach spaces.
In degree 1, the vanishing of cohomology of $\Gamma$ with coefficients in a representation $\pi$
translates to a fixed point property for affine actions of $\Gamma$ with linear part $\pi$. 
A classical example is Kazhdan's property (T), which for a group $\Gamma$ was characterized by
Delorme and Guichardet by vanishing of 
cohomology in degree 1 with coefficients in every unitary representation of $\Gamma$, see \cite{bhv}. 
Vanishing of cohomology with coefficients in representations on more general Banach spaces was studied and discussed 
in, e.g. \cite{bfgm,bader-gelander-monod,fisher-margulis,mimura,pn,pn-survey}.
Thus in degree 1 Theorem \ref{theorem: main theorem} yields a deformation result for actions with fixed points for finitely presented groups.
\begin{theorem}
Let $\Gamma$ be finitely presented with a finite generating set $S$. Let $H^2(\Gamma,\pi)$ be reduced and assume that every affine action of $\Gamma$ with linear part $\pi$ has a fixed point. Then there 
exists $\varepsilon=\varepsilon(\Gamma,S,\pi)>0$ such that if $\rho$ is an $\varepsilon$-deformation of $\pi$ then every affine action with linear part $\rho$ has a fixed point.
\end{theorem}

Here, similarly as before, we can replace finite presentability by merely finite generation if we assume that $L_pH^2(\Gamma,\pi)$ is reduced, instead of $H^2(\Gamma,\pi)$.
Since Kazhdan's property (T) is characterized by the existence of a fixed point for every affine isometric action, we obtain the following
\begin{theorem}\label{theorem : deforming higher property (T)}
Let $\Gamma$ be finitely presented with a finite generating set $S$.
If $\Gamma$ has property (T)
then for every unitary representation 
$\pi$ such that $H^2(\Gamma,\pi)$ is reduced there exists $\varepsilon=\varepsilon(\Gamma,S,\pi)>0$ such that for every $\varepsilon$-deformation
$\rho$ of $\pi$, every affine isometric action with linear part $\rho$ has a fixed point.

If $H^2(\Gamma,\pi)$ is reduced for every unitary representation $\pi$ then the $\varepsilon$ above can be chosen independently of $\pi$.
\end{theorem}

Theorem \ref{theorem : deforming higher property (T)} applies to automorphism groups of thick buildings,
which exhibit vanishing of higher cohomology with coefficients in any unitary representation under appropriate link conditions 
\cite{ballmann-swiatkowski,dymara-januszkiewicz,oppenheim}.

The above corollaries of Theorem \ref{theorem: main theorem} can be compared with the results of Fisher and Margulis \cite{fisher-margulis}.
A map $\varphi:X\to X$ of a metric space $X$ is an \emph{$\varepsilon$-isometry} if it satisfies the bi-Lipschitz condition
$$(1-\varepsilon)\Vert v-w\Vert \le \Vert \varphi(x)-\varphi(y)\Vert\le (1+\varepsilon) \Vert v-w\Vert.$$
Theorem 1.6 in \cite{fisher-margulis} states that for a locally compact, $\sigma$-compact group $\Gamma$ with property (T) and compact generating set $K$,
there exists $\varepsilon>0$, depending  on $\Gamma$ and $K$ only, such that for any continuous action of $\Gamma$ on
a Hilbert space, where $K$ acts by $\varepsilon$-almost isometries, there exists a fixed point.

Note that for a finitely generated group, if a representation $\rho$ is an $\varepsilon$-deformation of a unitary representation, then the generating set acts 
by $\varepsilon$-isometries. Thus our theorem in this case gives a statement similar to the one proved by Fisher and Margulis. However, our
methods have several advantages. First, the proof in \cite{fisher-margulis}
is indirect and does not give any insight into the value of $\varepsilon$ for which the theorem holds. 
Our Theorem \ref{theorem: main theorem} not only extends this type of phenomena to higher cohomology, but also gives explicit estimates of the $\varepsilon$ in terms of 
higher Kazhdan constants. Indeed, our arguments are effective and the behavior of various constants can be traced throughout the proofs.
Second, the argument used in \cite{fisher-margulis} is not available in the same generality as Theorem \ref{theorem: main theorem}.
Indeed, the argument in \cite{fisher-margulis} is based on an ultrapower construction and relies on the fact that the ultrapower of a Hilbert 
space is again a Hilbert space.
A similar fact for $L_p$-spaces was proved by Heinrich and Mankiewicz \cite{heinrich-mankiewicz}, however for general Banach spaces such methods are not available.

An important feature  of  Theorem \ref{theorem: main theorem} is that we obtain information about 
higher cohomology groups.
In particular, we can deform the top dimensional cohomology 
of a  group of geometric dimension $n$. 
Then we can conclude that deformations in dimension $n$ preserve the vanishing of cohomology.
\begin{corollary}
Let $\Gamma$ be a group of geometric dimension $n$ and a corresponding Eilenberg-MacLane space $X$.
If $H^{n}(\Gamma,\pi)=0$ then 
there exists $\varepsilon=\varepsilon(\Gamma,X,\pi)>0$ such that
$$H^n(\Gamma,\rho)=0$$ 
for any $\varepsilon$-deformation $\rho$ of $\pi$.
\end{corollary}

Another situation, mentioned earlier, in which $H^{n+1}(\Gamma,\pi)$ is automatically reduced is when $\Gamma$ is of type $F_{n+1}$ and $\pi$ is finite-dimensional. 
In that case we obtain the following
\begin{corollary}
Let $\Gamma$ be of type $F_{n+1}$ with a corresponding Eilenberg-MacLane space $X$ 
and let $\pi$ be a finite-dimensional representation of $\Gamma$. If $H^n(\Gamma,\pi)=0$ then 
there exists $\varepsilon=\varepsilon(\Gamma,X,\pi)>0$ such that
$$H^n(\Gamma, \rho)=0$$ 
for every $\varepsilon$-deformation $\rho$ of $\pi$.
\end{corollary}

The last application that we  derive is the following criterion for non-vanishing of cohomology.
\begin{corollary}
Let $\Gamma$ be of type $F_{n+1}$ and let $\pi$ be a representation such that $H^n(\Gamma,\pi)=0$ and there exists
a family $\pi_i$ of deformations converging to $\pi$ such that $H^n(\Gamma,\rho_i)\neq 0$. Then $H^{n+1}(\Gamma,\pi)\neq 0$.
\end{corollary}

We stress that all the arguments are effective and quantitative, and it is possible to estimate all the constants. 
We do not to compute the exact estimates and instead we indicate what do they depend on.

\subsection*{Acknowledgements} We would like to thank Roman Sauer and Shmuel Weinberger for helpful comments.
We are also grateful to Narutaka Ozawa for bringing \cite{burger-ozawa-thom} to our attention.
\tableofcontents

\section{Deformations of representations}\label{section : deformations}

\subsection{The space $\Hom(\Gamma, B_{inv}(E))$}
Consider a finitely generated group $\Gamma$ with a fixed finite generating set $S=S^{-1}$. Let $E$ be a Banach space and denote 
by $B_{inv}(E)$ the set of linear self-isomorphisms of $E$, equipped with the operator norm.

Given a subset $W\subseteq B_{inv}(E)$ by $\Hom(\Gamma,W)$ 
we will denote the space of homomorphisms $\varphi:\Gamma\to W$ (depending on $W$ it can happen that $\Hom(\Gamma,W)$ is empty).
We equip $\Hom(\Gamma,W)$ with the metric 
$$d_S(\pi,\rho)=\sup_{s\in S} \Vert \pi(s) - \rho(s)\Vert,$$
and the corresponding uniform  topology.

In particular, if $\operatorname{Iso}(E)$ denotes the group of isometries of a Banach space $E$, 
the space $\Hom(\Gamma, \operatorname{Iso}(E))$ is the space of isometric representations of $\Gamma$ on $E$.

\subsection{Deformations} 

\begin{definition}
Let $\Gamma$ be a group generated by a finite set $S$. Given a representation $\pi$ of $\Gamma$ on a Banach  space $E$ and $\varepsilon>0$, 
a representation $\rho$ of $\Gamma$ on  
$E$  is said to be an \emph{$\varepsilon$-deformation} of $\pi$ if 
$$d_S(\pi,\rho)\le \varepsilon.$$
\end{definition}
The most basic example of a deformation can be obtained by defining 
$$\rho_\gamma=T\pi_\gamma T^{-1},$$
where $T$ is an isomorphism of the Banach space on which $\pi$ is defined, with $\Vert I-T\Vert\le \varepsilon$ appropriately small.
However, in this case $T$ also induces an isomorphism of cohomology groups
$H^*(\Gamma,\pi)$ and $H^*(\Gamma,\rho)$. Thus Theorem \ref{theorem: main theorem}  for such deformations is trivially true and they are  not interesting for us. 

In the case $\Gamma=\mathbb{F}_n$ it is particularly easy  to construct deformations. Indeed, if $S$ is the free generating set then 
any map $S\to B_{inv}(H)$ extends to a representation of $\mathbb{F}_n$. Any two such maps that are within $d_S$ distance $\varepsilon$,
define representations which are deformations of each other.

Given a group $\Gamma$ with a $H$ and an appropriate representation of $H$ one can construct the induced representation of $\Gamma$.
Using this idea in \cite[Section 4]{burger-ozawa-thom} the authors constructed a family of pairwise non-equivalent deformations of the regular representation
of any group with a free subgroup.

Natural examples of  nontrivial deformations for other groups 
can also be obtained by twisting a representation by an appropriately chosen non-trivial cocycle,
which is small in norm on the generators.
Below we give examples of such deformations.

\begin{example}\label{example: Radon-Nikodym cocycle} \normalfont
Let $\Gamma$  act on measure space $(\Omega, \mu)$, with measure $\mu$ finite, by measure class preserving transformations. 
The action of $\Gamma$ induces a unitary representation $\pi$ on $L_2(\Omega,\mu)$,
$$\pi_g f=(g\cdot f)\left(\dfrac{d g\mu}{d \mu}\right)^{1/2}.$$
If the action of $\Gamma$ does not preserve the measure then
$$\rho_g f=(g\cdot f)\left(\dfrac{d g\mu}{d \mu}\right)^{\alpha},$$
where $\alpha\in (0,\infty)$, defines a family of deformations of $\pi$.
\end{example}

\begin{example}\label{example: twisting by a derivation}\normalfont
Let $\pi, \pi' $ be  representations on a Banach space $E$ and consider the collection of operators 
$$T_g=\left(\begin{array}{cc}\pi_g & D_g \\0 & \pi'_g\end{array}\right),$$
on $E\oplus E$.
It is easy to check that the $T_g$ form a representation of $\Gamma$ if and only if $D$ is a derivation  $D:\Gamma\to B(E)$, where $B(E)$ is considered to be a  
$\Gamma$-bimodule by composing with $\pi$ on the left and $\pi'$ on the right; that is, 
$D$ satisfies $D(gh)=\pi_gD(h)+D(g)\pi'_h$.
Given such a derivation $D:\Gamma\to B(E)$ we define 
$$\rho_g^{(\alpha)}=\left(\begin{array}{cc}\pi_g & \alpha D_g \\0 & \pi'_g\end{array}\right),$$
where $\alpha\in [0,\infty)$. It is easy to see that this gives a family of deformations of the representation
$$\rho^{(0)}_g=\left(\begin{array}{cc}\pi_g & 0 \\0 & \pi'_g\end{array}\right).$$
\end{example}

A representation $\pi$ on a Hilbert space $E$ is \emph{unitarizable} if it is similar to a unitary representation. 
 An obvious necessary condition for a representation to be unitarizable is that 
$\sup_{\gamma\in \Gamma}\Vert\pi_{\gamma}\Vert<\infty$. Note that the deformations of unitary representation discussed in 
Examples \ref{example: Radon-Nikodym cocycle} and \ref{example: twisting by a derivation} do not have to be uniformly bounded. In the case of 
Example \ref{example: twisting by a derivation} it is known (see \cite{pisier}) that $\rho^{(\alpha)}$ is unitarizable if and only if the derivation $D$ is inner; that is,
there exists $T\in B(H)$ such that $D_\gamma=\pi_\gamma T-T\pi'_\gamma$ for every $\gamma\in \Gamma$.

\section{Cohomology}

The general reference on cohomology of groups is \cite{brown}.
Let $\Gamma$ act  on $X$, which is a  uniformly locally finite simplicial complex, by simplicial automorphisms.  In $X$ there is a fundamental domain $F\subseteq X$, which we assume
to be a subcomplex.

\subsection{Finiteness conditions}
Recall that a group $\Gamma$ is said to be of type $\F_n$ if it admits an Eilenberg-Maclane space 
$B\Gamma$ which is a simplicial complex with a finite $n$-skeleton. $\Gamma$ is said to be of type $\F_{\infty}$ if
it is of type $\F_n$ for every $n\ge 1$.
See \cite{brown,geoghegan} for a discussion. Such finiteness properties will be important in our considerations.
The condition $\F_1$ is equivalent to the group $\Gamma$ being finitely generated, while $\F_2$ holds if and only if the group $\Gamma$ is finitely presented.
Examples of groups satisfying such finiteness conditions include e.g. hyperbolic groups, combable groups, Thompson groups, $\operatorname{SL}_n(\ZZ)$.

In the case when the group $\Gamma$ is of type $\F_n$, we can choose $X$ to be contractible with the fundamental domain $F$ above having 
finite skeletons $F^{(k)}$ for $k=0,1,\dots, n$, and with $\Gamma$ acting by simplicial automorphisms.
Let
$$S=\setd{g\in \Gamma}{\overline{F}\cap g\cdot \overline{F}\neq \emptyset}\subseteq \Gamma.$$
Under the condition $\F_1$ we have that $S$ is finite and $S$ generates $\Gamma$.
Our convention is that whenever we are considering a group $\Gamma$ of type $F_n$, then there is a chosen Eilenberg-MacLane space $X$ of $\Gamma$ 
witnessing this requirement; i.e.,
$X$ is a simplicial complex with  a finite $n$-skeleton.

\subsection{Cohomology with coefficients in a representation}
Let $X$ be a locally finite simplicial complex with a simplicial, free action of $\Gamma$ and let $F\subset X$ be a fundamental domain.
Let $\pi$ be a representation of $\Gamma$ on a Banach space $E$.
For a simplex $\sigma\in X$ we denote  by $\gamma_\sigma\in \Gamma$ the unique element $\gamma\in \Gamma$ such that  $\gamma_{\sigma}\cdot\sigma\in F$.
As usual, let $X^{(n)}$ denote the $n$-skeleton of a complex $X$; that is, the collection of all $n$-simplices of $X$. 

For $1\le p <\infty$ the space of $p$-integrable (untwisted) cochains is the space
$$C^n_{(p)}(X,\Gamma,E)=\setd{f:X^{(n)}\to E}{f \text{ alternating},  f(\sigma)=f(\gamma_{\sigma}\cdot\sigma), \sum_{\sigma \in F^{(n)}} \Vert f(\sigma)\Vert_E^p<\infty}.$$
For a simplex $\sigma=(v_0,\dots,v_k)$ denote $\sigma_i=(v_0,\dots, \hat{v}_i,\dots,v_k)$.
The codifferential is the map $d^n_{\pi}:C^n_{(p)}(X,\Gamma,E)\to C^{n+1}_{(p)}(X,\Gamma, E)$ given by the formula
$$ d^n_{\pi} f(\sigma)=\sum_{i=0}^n (-1)^i \pi_{g_{\sigma}^{-1}}\pi_{g_{\sigma_i}}f(\sigma_i).$$

If $\Gamma$ is of type $\F_k$ the space of cochains, $C^n_{(p)}(X,\Gamma,E)$, admits a Banach space structure for every $n\le k$.
\normalfont Indeed, the $n$-cochains $C^n_{(p)}(X,\Gamma,E)$ 
can be viewed as a subspace of the direct sum over the simplices of $F^{(n)}$, with a natural Banach space structure via the norm 
$$\Vert f\Vert_p=\left(\sum_{\sigma\in F^{(n)}} \Vert f(\sigma)\Vert_E^p\right)^{1/p}. $$
The dual space $(C^n_{(p)}(X,\Gamma,E))^*$ is then isometrically isometric with a quotient $C^n_{(q)}(X, \Gamma, E^*)$ of this direct sum, equipped with the norm 
$\Vert \cdot \Vert_{q}$, where $p^{-1}+q^{-1}=1$.

The codiferrentials are bounded linear operators for every $1\le p<\infty$ with respect to the above norms, see e.g., \cite{ballmann-swiatkowski,koivisto}.

The \emph{$L_p$-cohomology groups of $X$ with coefficients twisted by $\pi$} are defined as the cohomology of the above cochain complex:
$$L_pH^n(X,\Gamma,\pi)=\ker d^{n}_{\pi} \Big/\operatorname{im}d^{n-1}_{\pi}.$$
We can also  consider the  reduced cohomology,
$$L_p\Hred^n(X,\Gamma,\pi)=\ker d^{n}_{\pi}\Big/\ \overline{\operatorname{im}d^{n-1}_{\pi}},$$
where $\overline{\operatorname{im}d^{n-1}_{\pi}}$ denotes the norm closure of the image of $d^{n-1}_{\pi}$.
We say that $L_pH^n(X,\Gamma,\pi)$ is \emph{reduced} if $L_p\Hred^n(X,\Gamma,\pi)=L_pH^n(X,\Gamma,\pi)$. 

\begin{remark}\normalfont
The above definition of $L_p$-cohomology with coefficients in the $\Gamma$-module $(E,\pi)$ agrees with the one considered earlier in \cite{ballmann-swiatkowski,koivisto}.
Usually one considers twisted cochains $C^n(X,\pi)$, satisfying $f(\gamma\cdot \sigma)=\pi(g)f(\sigma)$ with the standard differential, 
$d^n f(\sigma)=\sum_{i=0}^n (-1)^i f(\sigma_i)$.
We then define the ``untwisting map'' $\varphi^n:C^n(X,\Gamma,\pi)\to C^n(X,\Gamma,E)$ by 
$$\varphi^n(f)(\sigma)=\pi(g_\sigma^{-1})f(\sigma)=f(g_\sigma^{-1}\cdot \sigma).$$
One can then check, that 
$$d^n_{\pi}\circ \varphi^{n+1}=\varphi^n\circ d^{n-1}.$$
\end{remark}

\begin{remark}\normalfont 
In the case when $X$ is uniformly locally finite we can also consider the $\ell_{\infty}$-cochains,
$$C^n_{(\infty)}(X,\Gamma,E)=\setd{f:X^{(n)}\to E}{f \text{ alternating},  f(\sigma)=f(\gamma_{\sigma}\cdot\sigma), \sup_{\sigma \in X^{(n)}} \Vert f(\sigma)\Vert_E<\infty}.$$
Under the assumption of uniform local finiteness of $X$ we have that the codifferential is a bounded linear operator.
\end{remark}

In the case when $X$ is contractible and the action of $\Gamma$ on $X$ is cocompact,
$$L_pH^n(X,\Gamma,\pi)=H^n(\Gamma,\pi),$$
is the group cohomology of $\Gamma$ with coefficients in the $\Gamma$-module $(E,\pi)$.

We have the following interpretation of the vanishing of $L_pH^n(X,\Gamma,\pi)$.
The fact that $L_pH^n(X,\Gamma,\pi)=0$ is, by definition, the surjectivity of the codifferential $$d^{n-1}_{\pi}:C^{n-1}_{(p)}(X,\Gamma,E)\to \ker d^n_{\pi}.$$ 
Consider the adjoint map $$\left(d^{n-1}_{\pi}\right)^*:\left(\ker d^n_\pi\right)^* \to C^{n-1}_{(p)}(X,\Gamma,E).$$
Standard duality arguments yield the following

\begin{proposition}\label{proposition: cohomology vanishes iff d* bounded below}
Let $\Gamma$ be a group acting on a uniformly locally finite simplicial complex $X$ and let $\pi$ be a representation of $\Gamma$ on a Banach space $E$. 
The following conditions are equivalent.
\begin{enumerate}
\item $L_pH^n(X,\Gamma,\pi)=0$,
\item There exists a constant $C>0$ such that $$\Vert (d_{\pi}^{n-1})^*\varphi\Vert\ge C\Vert \varphi\Vert$$ for every $\varphi\in (\ker d^n_\pi)^*$.
\end{enumerate}
\end{proposition}
The above proposition will be one of our main tools in the proof of Theorem \ref{theorem: main theorem}. A similar strategy was 
used in \cite{pn}.

\subsection{Higher Kazhdan constants}\label{section: higher Kazhdan constants}
Recall that for a group $\Gamma$, generated by a finite set $S$, the \emph{Kazhdan constant} is the number
$$K(\Gamma,S)=\inf\setd{ \sup_{s\in S} \dfrac{\Vert \pi_sv-v\Vert}{\Vert v\Vert}}{\pi \text{ unitary, with no non-zero invariant vectors} }.$$
A group $\Gamma$ has Kazhdan's property (T) if and only if $K(\Gamma,S)>0$.
We will define a higher Kazhdan-type constant in the context of cohomology. 
\begin{definition}
Let $\Gamma$ be a group acting on a uniformly locally finite simplicial complex $X$ and let $\pi$ be a representation of $\Gamma$ on a Banach space $E$. 
If  the cohomology $L_pH^{n+1}(X,\Gamma,\pi)$ is reduced
then the number 
$$\kappa_n(\Gamma,X,p,\pi)=\inf \setd{ \dfrac{\Vert d^n_\pi f\Vert }{\inf_{v\in \ker d^n_\pi}\Vert f+v\Vert}}{f\notin \ker d^n_\pi},$$
will be called the $n^{th}$-Kazhdan constant of the triple $(\Gamma,X,\pi)$.
\end{definition}

Another way to define the constant $\kappa_n(\Gamma,X,p,\pi)$ is as follows. Since the image of $d^n_{\pi}$ is closed in $C^{n+1}_{(p)}(X,\Gamma,E)$, 
the map $d^n_{\pi}$ descends to an isomorphism $\widetilde{d}^n_\pi$ between $C^n_{(p)}(X,\Gamma,E)/\ker d^n_{\pi}$ and $\operatorname{im} d^n_{\pi}\subseteq C^{n+1}_{(p)}(X,\Gamma,E)$. 
\begin{diagram}
C^n_{(p)}(X,\Gamma,E)&\rTo^{d^n_\pi} &\ker d^{n+1}_\pi\\
\dOnto &\ruTo(2,2)^{\widetilde{d}^n_\pi}& \dInto\\
C^n_{(p)}(X,\Gamma,E)/\ker d^n_{\pi}&\rTo^{\ \ \ \ \ \ \ \ \ \ \ \ \ \ \ }& C^{n+1}_{(p)}(X,\Gamma,E).
\end{diagram}
The constant $\kappa_n(\Gamma,X,p,\pi)$ is then the supremum of those constants $D>0$ that satisfy
$$\Vert \widetilde{d}^n_\pi v\Vert\ge D\Vert v\Vert,$$
for every $v\in C^n_{(p)}(X,\Gamma,E)/\ker d^n_\pi$.

Let now $\mathcal{P}$ be any family of representations of $\Gamma$ on a Banach spaces that is closed under taking infinite direct sums.
For instance, unitary representations, or more generally, uniformly bounded representations on a Hilbert space, whose norms are all bounded by a uniform constant, form such families.
On the other hand, finite-dimensional representations do not form such a family.
\begin{proposition}
Let $\Gamma$ act on a uniformly locally finite simplicial complex $X$. Let $n\ge 1$ and let $\mathcal{P}$ be as above. If $L_pH^n(X,\Gamma,\pi)$ is reduced for every 
representation $\pi\in\mathcal{P}$ then the $n$-th Kazhdan constant of $\mathcal{P}$
$$\kappa_n(\Gamma,X,p,\mathcal{P})= \inf\setd{\kappa(\Gamma,X,p,\pi)}{\pi \in \mathcal{P}}>0.$$
\end{proposition}
\begin{proof}
From the condition $L_pH^n(X,\Gamma,\pi)$ is reduced we have $\kappa(\Gamma,X,p,\pi)>0$ for every $\pi\in\mathcal{P}$.
Assume the contrary and let $\pi_n$ be a sequence of representations  in $\mathcal{P}$, such that 
$\kappa(\Gamma,X,p,\pi_n)\to 0$. Let $\pi=\oplus_{n\in \NN} \pi_n$ be a
representation on $\bigoplus_{n\in \NN} \mathcal{H}$. By the assumption on $\mathcal{P}$, $\pi\in\mathcal{P}$.
Then 
$$\kappa_n(\Gamma,X,p,\mathcal{P})\le \kappa_n(\Gamma,X,p,\pi_n),$$
for every $n\in \NN$.
\end{proof}
For instance, if $\mathcal{P}$ is the class of unitary representations of a finitely presented $\Gamma$, 
then $\kappa_0(\Gamma,X,\infty,\mathcal{P})$ is essentially the Kazhdan constant of $\Gamma$ (for the generating set determined by the action of $\Gamma$ on $X$).
If $\mathcal{P}$ is the class of unitary representations factoring through finite quotients then
$\kappa_n(\Gamma,X,\infty,\mathcal{P})>0$ if and only if  $\Gamma$ has property ($\tau$) \cite{lubotzky}. 

\section{Geometry of subspaces and quotients}

In order to analyze the cohomology with coefficient in a deformation we will need to look closely at the geometry of the spaces of cocycles and 
discuss a general setting for comparing norm bounds on operators on such spaces.
The following notion of close subspaces was studied in \cite{dymara-januszkiewicz} and used later in \cite{ershov-zapirain} in the context of Hilbert spaces.
Below we will consider this notion for subspaces of Banach spaces.

\begin{definition}
Let $E$ be a Banach space and $V$, $W$ be two subspaces of $E$. 
Let $\varepsilon>0$.
We say that $V$ is $\varepsilon$-close to $W$ if for every $v\in V$ there exists $w\in W$ such that
$$\Vert v-w\Vert\le \varepsilon\Vert v\Vert.$$
\end{definition}
We will not, in general, assume that $V$ and $W$ are closed.
If $E=\mathcal{H}$ is a Hilbert space then, if the subspace $W$ is closed in $\mathcal{H}$, in the above definition we can replace $w$ by $P_Wv$, where $P_W:\mathcal{H}\to W$ is the orthogonal projection onto $W$. In general, 
$w$ can be taken to be the (non-linear) nearest point projection.

\begin{lemma}
Let $V$ and $W$ be two closed subspaces of a Hilbert space $\mathcal{H}$.
\begin{enumerate}
\item If $V$ is $\varepsilon$-close to $W$ then $\Vert P_Wv\Vert\ge (1-\varepsilon) \Vert v\Vert$ for every $v\in V$;
\item if $\Vert P_Wv\Vert\ge c\Vert v\Vert$ for every $v\in V$ then $V$ is $\sqrt{1-c}$-close to $W$.
\end{enumerate}
\end{lemma}
\begin{proof}
The first claim follows after applying the triangle inequality 
$$\Vert v\Vert - \Vert P_Wv\Vert\le \Vert v- P_Wv\Vert\le \varepsilon\Vert v\Vert.$$
The second inequality is a consequence of the estimate
$$\Vert v-P_Wv\Vert^2=\Vert v\Vert^2-\Vert P_Wv\Vert^2\le (1-c)\Vert v\Vert^2.$$
\end{proof}

In the case of closed subspaces of Hilbert spaces a more detailed discussion on the $\varepsilon$-closeness and 
a related notion of $\varepsilon$-orthogonality can be found in \cite{ershov-zapirain}.

The following lemmas describe the behavior of bounded linear operators on close subspaces.

\begin{lemma}
Assume that $V, W$ are subspaces of a Banach space $E$. If $V$ is $\varepsilon$-close to $W$ then 
\begin{enumerate}
\item if $W\subseteq Y$ then $V$ is $\varepsilon$-close to $Y$; in particular,  $V$ is $\varepsilon$-close to $\overline{W}$;
\item $\overline{V}$ is $\varepsilon'$-close to $W$ for every $\varepsilon'>\varepsilon$.
\end{enumerate}
\begin{proof} This is obvious.
\end{proof}
\end{lemma}
\begin{lemma}\label{lemma: if operator T close to S on a subspace and S bounded below then T bounded below} 
Let $E$, $E'$ be Banach spaces and $V\subseteq E$ be a subspace. Let $T,S:E\to E'$ be linear bounded operators such that $\Vert T-S\Vert \le \varepsilon$.
If $\Vert Tv\Vert \ge C\Vert v\Vert$ for every $v\in V$ then
 $$\Vert Sv\Vert\ge (C-\varepsilon)\Vert v\Vert$$ 
for every $v\in V$.
\end{lemma}

\begin{proof}
We have
\begin{eqnarray*}
\Vert Sv\Vert&=&\Vert Sv-Tv+Tv\Vert\\
&\ge& \Vert Tv\Vert -\Vert (T-S)v\Vert\\
&\ge& C\Vert v\Vert - \varepsilon\Vert v\Vert.
\end{eqnarray*}
\end{proof}

\begin{lemma}\label{lemma: if two subspaces V,W are close and operator T bounded below on V then bounded below on W}
Let $E$, $E'$ be  Banach spaces, $V, W$ be closed subspaces of $E$. Let $T:E\to E'$ be a bounded linear operator. 
If $\Vert Tw\Vert\ge C\Vert w\Vert$ for every $w\in W$ and $V$ is $\varepsilon$-close to $W$ then 
$$\Vert Tv\Vert\ge \left( C -\varepsilon-\varepsilon \Vert T\Vert\right)  \Vert v\Vert,$$
for every $v\in V$.
\end{lemma}
\begin{proof}
Let $v\in V$ and let $w\in W$ be such that  $\Vert v-w\Vert\le \varepsilon\Vert v\Vert$. Then
\begin{eqnarray*}
\Vert Tv\Vert &=& \Vert Tv- Tw+Tw\Vert\\
&\ge & \Vert Tw\Vert-\Vert T(v-w)\Vert\\
&\ge &C\Vert w\Vert -\varepsilon\Vert T\Vert \Vert v\Vert\\
&\ge & C\Vert w-v+v\Vert -\varepsilon\Vert T\Vert \Vert v\Vert\\
&\ge & C(\Vert v\Vert-\Vert v-w\Vert)-\varepsilon\Vert T\Vert \Vert v\Vert\\
&\ge & \left( C -\varepsilon-\varepsilon \Vert T\Vert\right)  \Vert v\Vert.
\end{eqnarray*}
\end{proof}

\begin{lemma}\label{lemma : for close operators images are close subspaces}
Let $S,T:E\to E'$ be bounded linear operators between Banach spaces $E$ and $E'$. Let $\Vert T-S\Vert\le \varepsilon$ and assume that $S$ has closed range.
Then there exists a constant $\delta=\delta(\varepsilon, S)$, such that  $\im S$ is $\delta$-close to $\im T$ in $E'$.
\end{lemma}
\begin{proof}
Let $w=Sy\in \im S$. Since $S$ has closed image, there is a $C>0$ such that for every $ Sy$ there is $x\in E$ satisfying
$$Sy=Sx$$
and
$$\Vert Sx\Vert\ge C\Vert x\Vert.$$ 
Consider $Tx$. Then
$$\Vert Sx-Tx\Vert \le \Vert S-T\Vert\ \Vert x\Vert
\le \varepsilon C^{-1}\Vert Sx\Vert$$
\end{proof}

\begin{lemma}\label{lemma : kernel S is close to kernel T}
Let $S,T:E\to E'$ be bounded linear operators between Banach spaces $E$ and $E'$. Let $\Vert T-S\Vert\le \varepsilon$ and assume that $S$ has closed range.
Then there exists a constant $\delta=\delta(\varepsilon, S)$, such that  $\ker T$ is $\delta$-close to $\ker S$.
\end{lemma}
\begin{proof}
Consider $x\in \ker T$. Then 
$$\Vert Sx\Vert=\Vert Sx-Tx\Vert\le\varepsilon\Vert x\Vert.$$
Let $y\in $ be a vector satisfying 
$$d(0, \ker S+x)=\inf_{v\in \ker S} \Vert x+v\Vert=\Vert y\Vert.$$
We have 
$$K\Vert y\Vert\le \Vert Sy\Vert,$$
for some $K>0$, independent of $x$.
Then 
$$x-y\in \ker S,$$
and
$$\Vert x-(x-y)\Vert = \Vert y\Vert\le K^{-1}\Vert Sy\Vert= K^{-1}\Vert Sx\Vert \le K^{-1}\varepsilon\Vert x\Vert.$$
\end{proof}

We will also be interested in comparing operators on quotients spaces of a Banach spaces. The following lemma 
will be crucial in this context.
\begin{proposition}
Let $V,W\subseteq E$ be closed subspaces of a Banach space $E$, where $V$ is $\varepsilon$-close to $W$.
Let $w'$ be such that $\Vert w'\Vert_E\le \Vert w'+w\Vert_E$ for any $w\in W$ and consider the affine subspace
$V+w'$. Let $v'$ be such that 
$$\Vert v'\Vert\le \Vert v'+v\Vert$$ for all $v\in V$ and 
$$v'-w'\in V.$$
Then
$$\Vert v'\Vert \ge (1-2\varepsilon)\Vert w'\Vert.$$
\end{proposition}

\begin{proof}

In general we have
\begin{equation}\label{equation: w' estimate}
\Vert v'-w'\Vert \le \Vert v'\Vert +\Vert w'\Vert\le 2 \Vert w'\Vert,
\end{equation}
since by definition, $\Vert v'\Vert\le \Vert w'\Vert$.

Now assume that there is some $v'$ as above such that 
$$\Vert v'\Vert < (1-2\varepsilon)\Vert w'\Vert.$$
Since $V$ is $\varepsilon$-close to $W$, there exists $w\in W$ such that 
$$\Vert (v'-w')-w\Vert = \Vert v'-(w'+w)\Vert \le \varepsilon\Vert v'-w'\Vert.$$
We have
\begin{eqnarray*}
\Vert w'\Vert &\le  & \Vert w+w'\Vert \\
&\le & \Vert w+w'-v'\Vert+\Vert v' \Vert \\
&< & \varepsilon \Vert v'-w'\Vert +(1-2\varepsilon)\Vert w'\Vert\\
\end{eqnarray*}
From this it follows that 
$$\Vert w'\Vert +(2\varepsilon-1)\Vert w'\Vert <\varepsilon \Vert v'-w'\Vert,$$
so $$2\Vert w'\Vert <\Vert v'-w'\Vert.$$
However, by (\ref{equation: w' estimate}),
$$2\Vert w'\Vert < \Vert v'-w'\Vert \le 2\Vert w'\Vert,$$
which gives a contradiction.
\end{proof}

For a subspace $V\subseteq X$ and $x\in X$ denote by $[x]_V$  the affine subspace $V+x$.
The following statement gives a quantitative way of comparing operators on certain quotient spaces of Banach spaces.

\begin{proposition}\label{proposition: comparing operators on quotient spaces}
Let $V,W\subseteq E$ be two closed subspaces of a Banach space $E$. Assume that $V$ is $\varepsilon$-close to $W$.
Let $T:E/V\to E'$ and $S:E/W\to E'$ be bounded linear operators, and let $\widetilde{T}$, $\widetilde{S}:E\to E'$ denote their lifts to $E$.
Assume that $\Vert \widetilde{T}-\widetilde{S}\Vert\le \delta$. If $T$ is bounded below by $C$ then $S$ is bounded below by 
$$  C(1-2\varepsilon)-\delta.$$
\end{proposition}

\begin{proof}
Let $$S([x]_W)=\widetilde{S}w',$$
where $w'$ minimizes the distance to $[x]_W$.
Then 
\begin{eqnarray*}
\Vert \widetilde{S}w'\Vert&\ge &\Vert \widetilde{T}w'\Vert- \Vert \widetilde{S}w'-\widetilde{T}w'\Vert\\
&\ge &C \Vert v'\Vert_E - \delta\Vert w'\Vert_E,
\end{eqnarray*}
where $w'-v'\in V$ and $v'$ minimizes the distance to $[w']_V$.

Now, by the fact that $V$ is $\varepsilon$-close to $W$ and by the previous lemma we get 
\begin{eqnarray*}
\Vert \widetilde{S}w'\Vert &\ge &C(1-2\varepsilon) \Vert w'\Vert_E- \delta\Vert w'\Vert_E,
\end{eqnarray*}
which means
$$\Vert S[w']\Vert \ge \left( C(1-2\varepsilon)-\delta\right) \Vert [w']\Vert_{X/W}.$$
\end{proof}

\section{Cohomology of deformations}

We will now use the notion of close subspaces to study subspaces of Banach spaces related to cohomology.
For a $(k-1)$-simplex $\tau$ denote by $\operatorname{ind}_k(\tau)$ the number of $k$-simplices $\sigma$ such that $\tau\subset \sigma$.

\begin{lemma}\label{lemma: d_pi and d_rho are e-close}
Let $\pi$ be a representation of a finitely generated group $\Gamma$ on a Banach space $E$.
Let $\rho$ be an $\varepsilon$-deformation of $\pi$. Then 
$$\Vert d^n_{\pi}-d^n_{\rho}\Vert\le C(\Gamma, X,\pi,p,\varepsilon, n),$$
as operators  $C_{(p)}^n(X,\Gamma,E)\to C_{(p)}^{n+1}(X,\Gamma,E)$,
where $C(\Gamma, X,\pi,p,\varepsilon,n)\to 0$ as $\varepsilon\to 0$.
\end{lemma}
\begin{proof}
Let $f\in C^n_{(p)}(X,E)$.
\begin{eqnarray*}
\Vert d^n_\pi f-d^n_\rho f\Vert^p&=&\sum_{\sigma\in F^{(n)}} \left\Vert \sum_{i=0}^n (-1)^i \pi_{\gamma(\sigma)}^{-1}\pi_{\gamma(\sigma_i)}f(\sigma_i)-\sum_{i=0}^n (-1)^i 
\rho_{\gamma(\sigma)}^{-1}\rho_{\gamma(\sigma_i)}f(\sigma_i)\right\Vert^p\\ 
&\le &C' \sum_{\sigma\in F^{(n)}} \left( \left\Vert \sum_{i=0}^n (-1)^i 
\left(\pi_{\gamma(\sigma)}^{-1}-\rho_{\gamma(\sigma)}^{-1}\right) 
\pi_{\gamma(\sigma_i)}f(\sigma_i)\right\Vert^p\right.\\
&&\left.\ \ \ \ \ \ \ \ \ \ \ \ +\left\Vert \sum_{i=0}^n (-1)^i \left(\rho_{\gamma(\sigma)}^{-1}\pi_{\gamma(\sigma_i)}-\rho_{\gamma(\sigma)}^{-1}
\rho_{\gamma(\sigma_i)}\right)f(\sigma_i)\right\Vert^p\huge\right)\\ 
&\le &C'\sum_{\sigma\in F^{(n)}} 
\left\Vert \pi_{\gamma(\sigma)^{-1}}-\rho_{\gamma(\sigma)^{-1}}\right\Vert^p\ 
\left\Vert  \sum_{i=0}^n (-1)^i  \pi_{\gamma(\sigma_i)} f(\sigma_i)\right\Vert^p\\
&&\ \ \ \ \ \ \ \ \ \ \ \ +\left\Vert \rho_{\gamma(\sigma)}^{-1} \right\Vert^p\ \left\Vert \sum_{i=0}^n (-1)^i \left(\pi_{\gamma(\sigma_i)}-
\rho_{\gamma(\sigma_i)}\right)f(\sigma_i)\right\Vert^p,
\end{eqnarray*}
where $C'$ depends on $n$, $\pi$ and $p$.
We then have 
\begin{eqnarray*}
\sum_{\sigma\in F^{(n)}} \left\Vert \sum_{i=0}^n (-1)^i \left(\pi_{\gamma(\sigma_i)}-\rho_{\gamma(\sigma_i)}\right)f(\sigma_i)\right\Vert^p
&\le&C''\sum_{\sigma\in F^{(n)}} \sum_{i=0}^n \left\Vert  (\pi_{\gamma(\sigma_i)}-\rho_{\gamma(\sigma_i)})f(\sigma_i)\right\Vert^p\\
&\le&C''n \left(\sup_{\sigma\in F^{(n)}}  \left\Vert  \pi_{\gamma(\sigma_i)}-\rho_{\gamma(\sigma_i)}) \right\Vert^p\right)\ \left( \sum_{\sigma_i} \left\Vert f(\sigma_i)\right\Vert^p\right)\\
&\le&C''n  \left( \sup_{\tau \text{(n-1) -simplex}} \operatorname{ind}_k(\tau)\right)  \varepsilon \Vert f\Vert^p.
\end{eqnarray*}
Altogether, we obtain
\begin{eqnarray*}
\Vert d^n_\pi f -d^n_\rho f \Vert &\le &  \left( \sup_{\tau \text{(n-1) -simplex}} \operatorname{ind}_k(\tau)\right)  \varepsilon \Vert f\Vert+
\varepsilon \left(\sup_{\sigma\in F^{(n)}} \Vert \rho_{\gamma(\sigma)}\Vert \right)\Vert f\Vert\\
&\le& \Vert f\Vert\ C(\Gamma,X,\pi,n,\varepsilon,p),
\end{eqnarray*}
since
$$\Vert \rho_{\gamma}\Vert\le \Vert \pi_\gamma\Vert+\varepsilon.$$
\end{proof}

\subsection{Deformations of cocycles}
An important  point that we would like to make is that the fact that the cohomology group $H^{n+1}(\Gamma,\pi)$ is reduced together with the finiteness condition $\F_{n+1}$, 
allows to deform $n$-cocycles of a deformation $\rho$ 
of $\pi$ to cocycles for $\pi$.

\begin{theorem}\label{theorem: close representations have close cocycles}
Let $\Gamma$ be of type $F_{n+1}$ with a corresponding Eilenberg-MacLane space $X$ and let $\pi$ be a representation of $\Gamma$ on $E$.
Let $\varepsilon>0$ and let $\rho$ be an $\varepsilon$-deformation of $\pi$.
If the following two conditions hold:
\begin{enumerate}
\item $H^n(\Gamma,\pi)=0$,
\item $H^{n+1}(\Gamma,\pi)$ is reduced,
\end{enumerate} 
then there exists a constant $C=C(\Gamma,X,n,\pi,\varepsilon)$, 
such that $\ker d^n_{\pi}$  is $C$-close to $\ker d^n_{\rho}$.
\end{theorem}

\begin{proof}
We apply Lemmas \ref{lemma : kernel S is close to kernel T} and \ref{lemma: d_pi and d_rho are e-close} to $d^n_\pi$ and $d^n_\rho$.
The constant $K$ appearing in the proof of lemma \ref{lemma : kernel S is close to kernel T} in this setting is the $n$-Kazhdan constant $\kappa_n(\Gamma,X,p,\pi)$.
Therefore $\ker d^n_\rho$ is $C$-close to $\ker d^n_\pi$, for an appropriate choice of the constant $C$.
\end{proof}

In the case $n=1$ the above deformation theorem for cocycles has the following geometric interpretation. Given $\pi$ with $H^2(\Gamma,\pi)$ reduced,
any affine action with linear part $\pi$ is close on the generators to an affine isometric action with linear part $\rho$.

For instance, in the case of the free group $\mathbb{F}_n$ the above principle applies to any deformation since free groups have cohomological dimension 1 and 
cohomology with any coefficients vanishes in degrees 2 and higher.

\subsection{Proof of the main theorem}
We are now in the position to prove the main result of the paper.

\begin{theorem}[Theorem 1$^\prime$ in the introduction]
Let $\Gamma$ be a group acting on a uniformly locally finite simplicial complex $X$, $\pi$ be a representation of $\Gamma$ on a Banach space $E$ and $1\le p<\infty$.
Assume that
\begin{enumerate}
\item $L_pH^n(X,\Gamma,\pi)=0$,
\item $L_pH^{n+1}(X,\Gamma,\pi)$ is reduced.
\end{enumerate}
Then there exists a constant $\varepsilon=\varepsilon(\Gamma,X,\pi,n,p)>0$, such that 
for every $\varepsilon$-deformation $\rho$  of $\pi$ we have 
$$L_pH^n(X,\Gamma,\rho)=0.$$
\end{theorem}

\begin{proof}
We have the following diagrams, in which the top one is dual to the bottom one:
%\begin{diagram}
%			&						&	& (\ker d_{\pi}^n)^*	&			&			&\\
%			&						& \ldTo(3,2)^{(d^{n-1}_\pi)^*}	&				&\luOnto		&&			\\
%C^{n-1}_{(p)}(X,\Gamma,E)^*	& 						&	& 				&			&C^n_{(p)}(X,\Gamma,E)^*	&\\
%			&\luTo(3,2)_{(d^{n-1}_\rho)^*}     &	& 				&\ldOnto		&			&\\
%			&						&	& (\ker d_{\rho}^n)^*	&			&			&\\
%&			&						&	& \ker d_{\pi}^n		&			&			\\
%&			&		&	\ruTo(3,2)^{d^{n-1}_\pi}	&				&\rdInto		&			\\
%&C^{n-1}_{(p)}(X,\Gamma,E)	& 						&	& 				&			&C^n_{(p)}(X,\Gamma,E)	\\
%&			&\rdTo(3,2)_{d^{n-1}_\rho}      	&		& 				&\ruInto		&			\\
%&			&						&	& \ker d_{\rho}^n	&			&			\\
%\end{diagram}

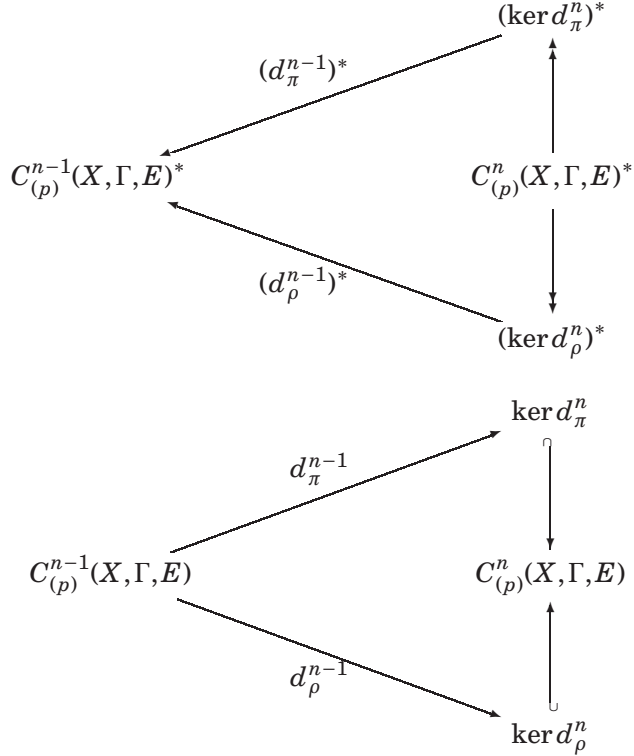
\begin{diagram}
						&				& (\ker d_{\pi}^n)^*\\
						&\ldTo(2,2)^{(d^{n-1}_\pi)^*}			&\uOnto\\
C^{n-1}_{(p)}(X,\Gamma,E)^*	&\ \ \ \ \ \ \ \ \ \ \ \ \ \ \ \ \ \ \ \ \ \ \ \ \ \ \ \ \ \ \ \ \ \ \ \ \ \ 			&C^n_{(p)}(X,\Gamma,E)^*	\\
						&\luTo(2,2)_{(d^{n-1}_\rho)^*} 				&\dOnto\\
						&				& (\ker d_{\rho}^n)^*\\
\end{diagram}

\begin{diagram}
						&				& \ker d_{\pi}^n\\
						&\ruTo(2,2)^{d^{n-1}_\pi}			&\dInto\\
C^{n-1}_{(p)}(X,\Gamma,E)	&\ \ \ \ \ \ \ \ \ \ \ \ \ \ \ \ \ \ \ \ \ \ \ \ \ \ \ \ \ \ \ \ \ \ \ \ \ \ 			&C^n_{(p)}(X,\Gamma,E)	\\
						&\rdTo(2,2)_{d^{n-1}_\rho} 				&\uInto\\
						&				& \ker d_{\rho}^n\\
\end{diagram}

We need to show that $d_\rho^*$ is bounded below.
Recall that the  space $(\ker d_\pi)^*$ is isometrically isomorphic with the quotient space $C^n_{(p)}(X,\Gamma,E)^*/\Ann(\ker d_\pi)$.
We need to consider several cases.

1) $d^{n-1}_\pi$ is the zero map (e.g., when $n=0$). Then  $L_pH^n(X,\Gamma,\pi)=0$ if and only if 
$$\ker d_\pi^n=\im d_\pi^{n-1}=0.$$
If, additionally, $\im d_n^{\pi}$ is closed then $d^n_{\pi}$ is bounded below  and lemma \ref{lemma: if operator T close to S on a subspace and S bounded below then T bounded below}
with $V=C^{n}_{(p)}(X,\Gamma,E)$ yields that $d_\pi^n$ is also bounded below. Consequently, $\ker d^n_\rho=0$, which implies $L_pH^n(X,\Gamma,\rho)=0$.

2) $d^n_\pi$ is the zero map. (This happens, e.g.,  when the dimension of the Eilenberg-MacLane space $X$ for the group $\Gamma$ is $n$).
In this case $(d^n_\pi)^*$ is bounded below on $C^n_{(p)}(X,\Gamma,E)^*$. By lemma \ref{lemma: if operator T close to S on a subspace and S bounded below then T bounded below},
also $(d^n_\rho)^*$ is bounded below on $C^n_{(p)}(X,\Gamma,E)^*$.

3) $d_\pi^n$ is not zero. Since the range of $d_\pi^n$ is closed, we have 
$$\Ann(\ker d_\pi^n)=\im (d_\pi^n)^*.$$
For $\rho$, on the other hand, we only have a dense inclusion $\im (d_\rho^n)^*\subseteq\Ann(\ker d_\rho^n)$.
Since $\rho$ is an $\varepsilon$-deformation of $\pi$ and the range of $(d_\pi^n)^*$ is closed, there is a constant $\eta\ge 0$, 
given by lemma \ref{lemma : for close operators images are close subspaces}, such that 
$$\im (d_\pi^n)^*\ \ \text{ is } \eta\text{-close to }\ \ \im (d_\rho^n)^*.$$
For the annihilators, this implies 
$$\Ann (\ker d_\pi^n)=\im (d_\pi^n)^* \ \ \text{ is } \eta\text{-close to }\ \ \Ann(\ker d_\rho^n)=\overline{\im (d_\rho^n)^*}.$$
We can now compare the operators $(d_\pi^n)^*$ and $(d_\rho^n)^*$. Indeed, since $(d^n_\pi)^*$ is bounded below on $C^n_{(p)}(X,E)^*/\Ann(\ker d^n_\pi)$,
then the above and lemmas \ref{lemma: d_pi and d_rho are e-close} and 
\ref{proposition: comparing operators on quotient spaces} yield that for sufficiently small $\varepsilon\ge 0$, the
operator $(d^n_\rho)^*$ is also bounded below on the quotient space $(\ker d^n_{\rho})^*= C^n_{(p)}(X,\Gamma, E)^*/\Ann(\ker d^n_\rho)$.

Finally, to compare the higher Kazhdan constants recall that 
$$d^n_\pi: C^{n-1}_{(p)}(X,\Gamma,E)/\ker d^{n-1}_\pi \to \ker d^n_\pi$$
is an isomorphism, where $\kappa_{n-1}(\Gamma,X,p,\pi)$ is the lower bound.
Therefore its adjoint, 
$$(d^n_\pi)^*: (\ker d^n_\pi)^*   \to \left(C^{n-1}_{(p)}(X,\Gamma,E)/\ker d^{n-1}_\pi\right)^*,$$
is also an isomorphism with the same lower bound. Proposition \ref{proposition: comparing operators on quotient spaces}
allows to estimate the lower bound on the operator 
$$(d^n_\rho)^*: (\ker d^n_\rho)^*   \to \left(C^{n-1}_{(p)}(X)/\ker d^{n-1}_\rho\right)^*,$$
which we now know is an isomorphism, and the Kazhdan constant satisfies
$$\kappa_{n-1}(\Gamma,X,p,\rho)\ge \kappa_{n-1}(\Gamma,X,p,\pi)- c(\Gamma,X,\pi,n,\varepsilon,p).$$
\end{proof}

\subsection{Proof in the Hilbert case}

The following proof of Theorem \ref{theorem: main theorem} is specific to representations on Hilbert spaces 
and uses the properties of the Laplacian. It was provided to us by Roman Sauer.

Given a complex of Hilbert spaces,
\begin{diagram}
\dots &\rTo& C_{n-1}&\rTo^{d^{n-1}}& C_n &\rTo^{d_n}&C_{n+1}&\rTo&\dots
\end{diagram}
consider the Laplacian, 
$$\Delta^n=d^{n-1}(d^{n-1})^*+(d^n)^*(d^n):C_n\to C_n$$
and recall that $\Delta^n$ is a chain map. Additionally, $d^*$ is a chain homotopy between $\Delta^n$ and $0$, by definition of $\Delta^n$.
These two facts imply that there is an induced map $\overline{\Delta}^n$ on the cohomology $H^n$ of the above chain complex, and that $\overline{\Delta}^n=0$.

\begin{theorem}
Assume that $H^{n+1}$ is reduced. Then $H^n=0$ if and only if $\Delta^n$ is invertible.
\end{theorem}
\begin{proof}
Assume that $\Delta^n$ is invertible. Without loss of generality we can assume that $H^{n+1}=0$, by considering 
the chain complex truncated to $\operatorname{\im} d^n$, which is a closed subspace of $C^{n+1}$, by assumption.
Let $x\in C^n$ be such that $d^nx=0$ and take $y\in C^n$ satisfying $\Delta^n y=x$. We have
$$\Delta^n d^n y= d^n\Delta^n y =0,$$
since $\Delta^n$ is a chain map. Since $\ker \Delta^n=0$, we conclude that $d^ny=0$.

Now, since $\overline{H}^n=0$, we can choose a sequence $\set{z_i}$, satisfying
$$d^n z_i\to y.$$
Observe that 
$$d^n(d^n)^*d^nz_i \to d^n(d^n)^* y.$$
However, on the other hand we have
$$\Delta^n d^n z_i=d^n \Delta^n z_i=d^n(d^n)^*d^n z_i+d^n d^{n-1}(d^{n-1})^*z_i.$$
The last term vanishes and we obtain 
$$d^n(d^n)^*d^nz_i \to \Delta^n y=x.$$
It follows that 
$$x=d^n\left((d^n)^* y  \right),$$
which implies vanishing of $H^n$.

To prove the converse assume $H^n=0$ and argue to show that $\Delta^n$ is invertible.
Denote $Z^n=\ker d^n$ and $B^n=\im d_{n-1}$. Assume that $H^{n+1}$ is reduced and that $H^n=0$. This means that 
$B^{n+1}$ is closed in $C^{n+1}$ and that $B^n=Z^n$. In particular, this implies that $B^n$ is closed.

Consider the decompositions,
\begin{eqnarray*}
C^{n-1}&=& Z^{n-1}\oplus (Z^{n-1})^{\perp},\\
C^n&=& Z^n\oplus (Z^n)^{\perp}\ =\ B^n\oplus (B^n)^{\perp},\\
C^{n+1}&=& B^{n+1}\oplus (B^{n+1})^{\perp}.
\end{eqnarray*}
We have that $(d^n)^*d^n$ restricted to $Z^n$ is the zero map. We claim that $(d^n)^*d^n$ is invertible on $(Z^n)^{\perp}$.
Indeed, by the Open Mapping Theorem, $$d^n:(Z^n)^{\perp}\to B^{n+1}$$ is an isomorphism. Since $(d^n)^*$ restricted to ${B^{n+1}}^{\perp}$ is $0$,
it follows that $(d^n)^*:B^{n+1}\to (Z^n)^{\perp}$ is an isomorphism.

Now observe that $(d^{n-1})(d^{n-1})^*=0$ on $(B^n)^{\perp}$. We claim that $(d^{n-1})(d^{n-1})^*$ restricted to $B^n$ is invertible. 
Indeed, again by the Open Mapping Theorem, $d^{n-1}:(Z^{n-1})^{\perp}\to B^n$ is an isomorphism and it follows that $(d^{n-1})^*:B^n\to (Z^{n-1})^{\perp}$ is 
an isomorphism. It follows that $\Delta^n$ is invertible.

Applying the above argument to the cochain complex $C_{(2)}^n(X,\Gamma,E)$ with a given representation $\pi$ yields the assertion.
\end{proof}

\subsection{Assumptions on $H^{n+1}$ are essential}\label{subsection:H^n+1 must be reduced}
We will now show that the assumption that $H^{n+1}(\Gamma,\pi)$ being reduced is essential.

We will show this in degree 0 where cohomology is simply the subspace of invariant vectors of the representation. The examples we have in mind arise 
as representations without non-zero invariant vectors, but  with almost invariant vectors. However, we additionally want the almost invariant vectors to 
arise as invariant vectors for deformations. An explicit example of this type is exhibited below.

Consider the infinite cyclic group $\ZZ$. This group acts on $L_2(S^1)$, where the circle $S^1$ is viewed as the Pontraygin dual of $\ZZ$.
The representation of $\ZZ$ on $L_2(S^1)$ is given by specifying the generator:
$$Tf(z)= e^{iz}f(z).$$
This representation does not have non-zero invariant vectors; that is, $H^0(\Gamma,\pi)=0$.

We will now show that for every $\varepsilon>0$ there is an $\varepsilon$-deformation  $\pi_{\varepsilon}$ of $\pi$ for which 
$H^0(\Gamma,\pi_\varepsilon)$ does not vanish. Choose a neighborhood $U$ of positive but sufficiently small measure around the identity in $S^1$ and define 
$$T_\varepsilon f(z)=\left\{\begin{array}{ll} 
Tf(z)& \text{ if } z \notin U,\\
f(z)& \text{ if } z\in U.
\end{array}\right.$$
Choosing the neighborhood $U$ to have sufficiently small measure we ensure that 
$$\Vert T-T_\varepsilon\Vert\le \varepsilon,$$
which means that the representation $\pi_\varepsilon$ generated by $T_\varepsilon$ is an $\varepsilon$-deformation of $\pi$. 
On the other hand, $\pi_\varepsilon$ has non-zero invariant vectors, as long as $U$ has positive measure. That is,
$H^0(\ZZ,\pi_\varepsilon)\neq 0$, as claimed.

Note, that $H^1(\ZZ,\pi)$ is not reduced. Indeed, if $\rho$ is an $\varepsilon$-deformation of $\pi$ and $\rho$ has an invariant vector $v$, then 
$v$ is $\varepsilon$-invariant for $\pi$, in the sense that 
$$\Vert \pi_sv-v\Vert\le \varepsilon.$$
Thus $\pi$ has a sequence of almost invariant vectors but no non-zero invariant vectors, i.e., the range of $d_\pi^0$ is not closed.

%---------------------------------------------------

\section{Final remarks}
\subsection{Higher property (T) and reduced cohomology}\label{subsection:higher T}
By a cohomological characterization of property (T) one usually means a theorem of  Delorme and Guichardet: $\Gamma$ has property (T) if and only 
if $H^1(\Gamma,\pi)=0$ for every unitary representation $\pi$.  
However, there is another way to formulate property (T) in terms of cohomology,
that fits very well with our setup. 
\begin{proposition}
The following conditions are equivalent for a finitely generated group:
\begin{enumerate}
\item $\Gamma$ has property (T);
\item for every unitary representation $\pi$, $H^1(\Gamma,\pi)$ is reduced.
\end{enumerate}
\end{proposition}
Indeed, $H^0(\Gamma,\pi)$ consists of the $\pi$ invariant vectors on the representation space and the image of the first codifferential is closed precisely when 
there are no almost invariant vectors. Thus the above proposition is a cohomological reformulation of Kazhdan's original definition. The Delorme-Guichardet theorem 
states that $H^1(\Gamma,\pi)$ is reduced for \emph{every} unitary $\pi$ if and only if it vanishes for every $\pi$.

A generalization property (T) to higher dimensions can now be done in two ways. The first one amounts to requesting that $H^n(\Gamma,\pi)$ vanishes 
for every unitary representation. Examples of groups satisfying such conditions for $n\ge 2$ were discussed in \cite{ballmann-swiatkowski,dymara-januszkiewicz,oppenheim}.
The second route is via the following
\begin{definition}\label{definition: higher (T)}
Let $\Gamma$ be a group of type $F_{n+1}$. We say that $\Gamma$ has  property (T$_n$\,) if 
for every unitary representation $\pi$, $H^{n+1}(\Gamma,\pi)$ is reduced.
\end{definition}
Kazhdan's property (T) is then property (T$_0$).
It is clear that, similarly as in the classical case, the vanishing of $H^{n+1}(\Gamma,\pi)$ for every unitary $\pi$ trivially implies  property $(T_n)$.
We do not know if the converse is true. If the answer is positive, such an equivalence would provide a higher-dimensional analog of the Delorme-Guichardet theorem.

Interestingly, groups satisfying the conditions of Definition \ref{definition: higher (T)} fit perfectly into our framework and for these groups the
conditions of Theorem \ref{theorem: main theorem} are automatically satisfied.

\subsection{Weil's local rigidity criterion}
Let $\Gamma$ be a finitely generated group with a generating set $S$ and let $G$ be a Lie group. Consider the set of homomorphisms, $\Hom(\Gamma,G)$.
A homomorphism $\varphi\in \Hom(\Gamma,G)$ is said to be locally rigid if there exists $\varepsilon>0$ such that 
for any homomorphism $\psi\in \Hom(\Gamma,G)$ satisfying 
$$\sup_{s\in S} d_G(\varphi(s),\psi(s))\le \varepsilon,$$
there exists $g\in G$ such that 
$$\psi(\gamma)=g^{-1}\varphi(\gamma)g.$$
Weil \cite{weil1,weil2} proved the following criterion for local rigidity.

\begin{theorem}[A. Weil]
For $\varphi$ as above, if $H^1(\Gamma,\operatorname{Ad}\circ \varphi)=0$ then $\varphi$ is locally rigid.
\end{theorem}

Now we will apply Theorem \ref{theorem: main theorem}. Consider a homomorphism $\varphi\in \Hom(\Gamma, G)$ and
assume that the condition in Weil's theorem is satisfied for $\varphi$: $H^1(\Gamma,\operatorname{Ad}\circ \varphi)=0$. Note that the
representation $\operatorname{Ad}\circ \varphi$ is a representation on the vector space underlying the Lie algebra $\mathfrak{G}$. Since the $G$ is finite dimensional,
so is $\mathfrak{G}$ and it follows that $H^i(\Gamma,\operatorname{Ad}\circ \varphi)$ is reduced for all $i$. Therefore, by Theorem \ref{theorem: main theorem}, 
there exists a $\varepsilon_\varphi>0$, such that
for any $\psi\in \Hom(\Gamma,G)$, satisfying 
$$\sup_{s\in S}d_G(\varphi(s),\psi(s))\le \varepsilon_\varphi,$$
we have $H^1(\Gamma,\operatorname{Ad}\circ \psi)=0$. 
\begin{theorem}
Let $\Gamma$, $G$ be as above, where $\Gamma$ is finitely presented. The set 
$$W=\setd{\varphi\in\Hom(\Gamma,G)}{ H^1(\Gamma,\operatorname{Ad}\circ \varphi)=0}$$
is open in the uniform topology.
\end{theorem}

\subsection{Vanishing vs connectedness of $\Hom(\Gamma,U(\mathcal{H}))$ }

Consider now $\Hom(\Gamma,U(\mathcal{H}))$, the space of homomorphisms from $\Gamma$ into the
unitary group of a Hilbert space $\mathcal{H}$. Theorem \ref{theorem: main theorem} yields the following consequence about the metric and topological structure
of $\Hom(\Gamma,U(\mathcal{H}))$.

Throughout this section we assume that $\Gamma$ is a group of type $F_n$ and cohomological dimension $n$. We have a natural decomposition
$$\Hom(\Gamma,U(\mathcal{H}))=N_{\Gamma,\mathcal{H}}\sqcup V_{\Gamma,\mathcal{H}},$$
where
$$V_{\Gamma,\mathcal{H}}=\setd{ \pi\in\Hom(\Gamma,U(\mathcal{H}))}{ H^{n}(\Gamma,\pi)=0},$$
and
$$N_{\Gamma,\mathcal{H}}=\setd{ \pi\in\Hom(\Gamma,U(\mathcal{H}))}{ H^{n}(\Gamma,\pi)\neq 0}.$$

\begin{proposition}
The set $N_{\Gamma,\mathcal{H}}$ is closed.
\end{proposition}

\begin{proof}
Let $\pi_n$ be a sequence of elements of $N_{\Gamma,\mathcal{H}}$ converging to $\pi\in \Hom(\Gamma,U(\mathcal{H}))$. 
If $H^n(\Gamma,\pi)=0$ then, by Theorem \ref{theorem: main theorem}, we would have $H^n(\Gamma,\pi_n)=0$ for $\pi_n$ for $n$ sufficiently large.
\end{proof}

The set $V_{\Gamma,\mathcal{H}}$ is consequently open. However, under certain conditions $V_{\Gamma,\mathcal{H}}$ is also open.
Let $$\kappa\left(S,V_{\Gamma,\mathcal{H}}\right)=\inf_{\pi\in V_{\Gamma,\mathcal{H}}} \kappa_n(\Gamma,S,\pi).$$

\begin{proposition}\label{proposition: Kazhdan constants ==> N clopen}
If $\kappa(S,V_{\Gamma,\mathcal{H}})>0$ then $V_{\Gamma,\mathcal{H}}$ is closed.
\end{proposition}
\begin{proof}
Let $\pi_n\in V_{\Gamma,\mathcal{H}}$ and $\pi_n\to \pi$. By the previous estimate, if 
$\rho$ is a $\varepsilon$-deformation of $\pi_n$ for 
$$\varepsilon<\dfrac{\kappa_{n-1}(\Gamma,S,\pi)}{K_{n-1}},$$
then $H^n(\Gamma,\rho)=0$. Since the constant $K_{n-1}$ is independent of $\pi_n$, by the assumption on $\kappa(S,V_{\Gamma,\mathcal{H}})$
we conclude that there exists $n$ sufficiently large such that $\pi$ is an $\varepsilon$-deformation for 
$$\varepsilon< \dfrac{\kappa(S,V_{\Gamma,\mathcal{H}})}{K_{n-1}}.$$
\end{proof}

Recall that the Delorme-Guichardet theorem states that if the Kazhdan constant of $\Gamma$ is positive then $H^1(\Gamma,\pi)=0$ for every 
unitary representation $\pi$.
The following corollary can be regarded as a top-dimensional relative of the Delorme-Guichardet theorem.

\begin{theorem}
Let $\Gamma$ be such that $\kappa(S,V_{\Gamma,U(\mathcal{H})})>0$. Then vanishing of $H^n(\Gamma,\pi)$ is stable on connected components of
$ \Hom(\Gamma,U(\mathcal{H}))$. 

More precisely, let $\mathcal{C}\subseteq \Hom(\Gamma,U(\mathcal{H}))$ be a connected component.
The following conditions are equivalent:
\begin{enumerate}
\item  $H^n(\Gamma,\pi)$ vanishes for some  $\pi\in\mathcal{C}$.
\item  $H^n(\Gamma,\pi)$ vanishes for all $\pi\in\mathcal{C}$.
\end{enumerate}
\end{theorem}

\end{document}